\documentclass[12pt,twoside]{amsart}

\usepackage{amssymb}
\usepackage[OT1]{fontenc}
\usepackage[left=3.2cm,top=3.8cm,right=3.2cm]{geometry}
\usepackage{comment}
\usepackage{enumerate}

\usepackage[usenames]{color}
\numberwithin{equation}{section}

\usepackage[active]{srcltx}  

\newcommand{\N}{\mathbb{N}}         
\newcommand{\R}{\mathbb{R}}         

\newcommand{\PP}{\mathbb{P}}        
\newcommand{\supp}{\text{supp}}        

\newcommand{\e}{\varepsilon}
\newcommand{\wh}{\widehat}
\newcommand{\wt}{\widetilde}

\newcommand{\be}{\begin{equation}}
\newcommand{\ee}{\end{equation}}

\DeclareMathOperator{\hdim}{dim_H}
\DeclareMathOperator{\fdim}{dim_F}

\newtheorem{thm}{Theorem}[section]
\newtheorem{lemma}[thm]{Lemma}
\newtheorem{prop}[thm]{Proposition}
\newtheorem{cor}[thm]{Corollary}

\theoremstyle{remark}
\newtheorem{remark}[thm]{Remark}

\title[Self-similar measures: Fourier decay and dimension]{Self-similar measures: asymptotic bounds for the dimension and Fourier decay of smooth images}

\author{Carolina A. Mosquera}
\author{Pablo S. Shmerkin}
\address{
Departamento de Matem\'atica, Facultad de Ciencias Exactas y Naturales, Universidad de Buenos Aires \\
Ciudad Universitaria, Pabell\'{o}n I (C1428EGA) \\
 Ciudad de Buenos Aires, Argentina and IMAS-CONICET, Argentina}
\email{mosquera@dm.uba.ar}
\address{
        Department of Mathematics and Statistics\\
        Torcuato di Tella University and CONICET \\
        Av. Figueroa Alcorta 7350 (C1428BCW) \\
        Ciudad de Buenos Aires, Argentina}
 \email{pshmerkin@utdt.edu}

\thanks{CM and PS were supported by UBACyT 20020130100403BA, PICT 2014-1480 (ANPCyT) and PIP 11220150100355}

\begin{document}

\maketitle

\begin{abstract}
R. Kaufman and M. Tsujii proved that the Fourier transform of self-similar measures has a power decay outside of a sparse set of frequencies. We present a version of this result for homogeneous self-similar measures, with quantitative estimates, and derive several applications: (1) non-linear smooth images of homogeneous self-similar measures have a power Fourier decay, (2) convolving with a homogeneous self-similar measure increases correlation dimension by a quantitative amount, (3) the dimension and Frostman exponent of (biased) Bernoulli convolutions tend to $1$ as the contraction ratio tends to $1$, at an explicit quantitative rate.
\end{abstract}

\section{Introduction and statement of results}

Given a finite Borel measure on $\R$, let
\[
\widehat{\mu}(\xi) = \int e^{2 \pi i \xi x}\,d\mu(x)
\]
be its Fourier transform. The decay properties of $\widehat{\mu}(\xi)$ as $|\xi|\to\infty$ give crucial ``arithmetic'' information about $\mu$. Indeed, define the \emph{Fourier dimension} of $\mu$ as
\begin{equation}
\label{eq:power-decay}
\fdim(\mu) =   2\sup\{  \sigma\ge 0: |\widehat{\mu}(\xi)| \le C_\sigma |\xi|^{-\sigma} \text{ for some $C_\sigma>0$ }\}.
\end{equation}
If $\fdim(\mu)>0$, then $\mu$-almost all points are normal to all integer bases (in fact, for this some appropriate logarithmic decay is enough; this follows from \cite{DavenportEtAl63}). As another example, if $\fdim(\mu)>0$ and additionally $\mu$ satisfies a Frostman condition $\mu(B(x,r))\le C\,r^t$, then $\mu$ satisfies a restriction theorem analog to the Stein-Tomas theorem for the sphere, see \cite{Mockenhaupt00}.

Despite its importance, it is notoriously difficult to calculate, and often even to estimate, the Fourier dimension of natural measures, such as measures arising from some dynamical system (it is often easier for random measures). See \cite{JordanSahlsten16, BourgainDyatlov17} for some deep recent results in this direction. Moreover, it may well happen that there is no decay at all. For example, it is well known that if $\mu$ is any measure supported on the middle-thirds Cantor set $C$,  then $\wh{\mu}(\xi)\not\to 0$ as $\xi\to\infty$; in particular, $\fdim(\mu)=0$. This applies, in particular, to the Cantor-Lebesgue measure $\nu$ on $C$. Despite this, in \cite{Kaufman84}, Kaufman proved that if $F$ is any $C^2$ diffeomorphism of $\R$ such that $F''>0$, then $\fdim(F\nu)>0$ where, here and below, we define
\[
F\mu(A) = \mu(F^{-1}A)\quad\text{for all Borel sets }A\subset\R.
\]
Let $\nu_a$ be the distribution of the random sum
\begin{equation} \label{eq:random-sum}
\sum_{n=1}^\infty \pm a^n,
\end{equation}
where the signs are IID with equal probabilities. This is the well-known family of Bernoulli convolutions, and the special case $a=1/3$ yields the Cantor-Lebesgue measure $\nu$. In fact, Kaufman proved his result for all $\nu_a$ with $a\in (0,1/2)$. Kaufman's paper appears to be little known, and is tersely written. One of the goals of this paper is to clarify, extend and quantify Kaufman's argument. We prove that if $\mu$ is  any non-atomic self-similar measure on $\R$ arising from a homogeneous iterated function system, then $\fdim(F\mu)>c(\mu)>0$ for any $C^2$ diffeomorphism of $\R$ with $F''>0$ (see Theorem \ref{thm:kaufman}). This holds even if the construction of $\mu$ involves overlaps and, in particular, extends to all Bernoulli convolutions $\nu_a$ with $a\in (0,1)$. Moreover, the value of $c(\mu)$ is effective (although certainly far from optimal) and depends continuously on the parameters defining $\mu$. For example, for the Cantor-Lebesgue measure $\nu$ we calculate that
\[
\fdim(F\nu) \ge 0.016 \,.
\]
See Corollary \ref{cor:Fmu}.

Even if $\fdim(\mu)=0$, it may happen that $\wh{\mu}(\xi)$ has fast decay outside of a very sparse set of frequencies. Indeed, Kaufman's result mentioned above depends on proving this for the measures $\nu_a$, $a\in (0,1/2)$. He did this by using a variant of what has become known as the Erd\H{o}s-Kahane argument (the original application of this method, by Erd\H{o}s \cite{Erdos40} and Kahane \cite{Kahane71}, was to show that $\fdim(\nu_a)>0$ for all $a$ outside of zero Hausdorff-dimensional set of exceptions). Recently, Tsujii \cite{Tsujii15} (using different arguments) proved a similar result for arbitrary self-similar measures on $\R$:
\begin{thm}[Kaufman, Tsujii]
Let $\mu$ be a self-similar measure on the real line which is not a single atom. Then, given $\e>0$, there exists $\delta>0$ such that for all sufficiently large $T$, the set
\[
\{ \xi\in [-T,T]: |\wh{\mu}(\xi)|\ge T^{-\delta} \}
\]
can be covered by $T^\e$ intervals of length $1$.
\end{thm}
We note that Tsujii does not give any explicit estimates. In this work, we make the dependence of $\delta$ on $\e$ explicit for self-similar systems arising from \emph{homogeneous} iterated function systems.

Recall that $\nu_a$ denotes the Bernoulli convolution with parameter $a$. We consider also biased Bernoulli convolutions $\nu_a^p$, defined as the distribution of the random sum \eqref{eq:random-sum}, with the signs still IID, but now with $\PP(+)=p$, $\PP(-)=1-p$. It is generally believed that $\nu_a$, and indeed $\nu_a^p$ when $p$ is bounded away from $0$ and $1$, should become increasingly smooth as $a\uparrow 1$. In particular, one of the main conjectures in the field is that $\nu_a$  is absolutely continuous for all $a$ sufficiently close to $1$. Although this remains wide open, very recently new evidence has been established in this direction. P. Varju \cite{Varju16} proved that $\nu_a^p$ is absolutely continuous provided $a$ is algebraic and close enough to $1$ in terms of $p$ and its Mahler measure (in particular, Varju provided the first explicit examples of absolutely continuous biased Bernoulli convolutions). On the other hand, Shmerkin \cite{Shmerkin14} (building on work of Hochman \cite{Hochman14}) proved that $\nu_a^p$ is absolutely continuous for all $a\in (e^{-h(p)},1)$, outside of a set of possible exceptions of Hausdorff dimension $0$, which is independent of $p$. Here, and later,
\[
h(p)=-p\log p-(1-p)\log(1-p)
\]
is the entropy of the vector $(p,1-p)$. This was improved in \cite{Shmerkin16} by showing that, again outside of a zero-dimensional set of exceptions, $\nu_a^p$ has a density in $L^q$ for an optimal range of finite values of $q$. In this paper, we prove a weaker result which is nevertheless valid for \emph{all} $a$ close to $1$. To state the theorem, define the Frostman exponent or $L^\infty$-dimension of a measure $\mu$ on $\R$ as
\[
\dim_\infty(\mu) = \liminf_{r\downarrow 0} \frac{\log \sup_{x\in\supp(\mu)} \mu(B(x,r))}{\log r}.
\]
In other words, $\dim_\infty(\mu)$ is the supremum of all $s$ such that $\mu(B(x,r)) \le r^s$ for all $x$ and all sufficiently small $r$ (depending on $s$).
\begin{thm} \label{thm:dim-Bernoulli-intro}
For every $p_0\in (0,1/2)$ there is $C=C(p_0)>0$ such that
\[
 \inf_{p\in [p_0,1-p_0]}\dim_\infty(\nu_a^p) \ge 1 - C (1-a) \log(1/(1-a)).
\]
\end{thm}
We make some remarks on this statement:
\begin{enumerate}
\item It follows from Frostman's Lemma that the same lower bounds hold for other popular notions of dimension of a measure, including Hausdorff and correlation dimensions.
\item The value of $C(p_0)$ is effective in principle, so the result gives concrete lower bounds for Bernoulli convolutions for parameters close to $1$.
\item It is easy to obtain this bound for unbiased Bernoulli convolutions, by considering the smallest value of $n$ such that $(\nu_{a^n})$ satisfies the open set condition, and using that $\nu_a = \nu_{a^n}*\eta$ for a suitable measure $\eta$, the Frostman exponent does not decrease under convolution, and the value for $\dim_\infty(\nu_{a^n})$ is well known (this argument appears to be folklore). This elementary argument does not extend to the biased case. In the unbiased case, we are able to improve on the folklore argument, see Corollary \ref{cor:unbiased}.
\item The fact that $\dim_\infty(\nu_a^p)$ tends to $1$ as $a\uparrow 1$ (uniformly for $p\in [p_0,1-p_0]$) can also be deduced from the results of \cite{Shmerkin16}, but the resulting bounds (which would take some effort to make explicit) are in any case substantially worse.
\end{enumerate}

We deduce Theorem \ref{thm:dim-Bernoulli-intro} from another result asserting that convolving with a Bernoulli convolution increases correlation dimension by a quantitative amount; see Theorem \ref{thm:flattening} for details. (A non-quantitative version could also be deduced from \cite{Shmerkin16}.)

\section{Fourier decay outside of a small set of frequencies}

Let $\mu_{a,t}^p$ be the self-similar measure for the IFS $\{ a x+t_i\}_{i=1}^m$ with weights $p=(p_1,\ldots,p_m)$, where
$t=(t_1,\ldots,t_m)$, $a\in (0,1)$. In other words, $\mu_{a,t}^p$ is the distribution of the random sum $\sum_{n=1}^\infty X_n a^n$, where $X_n$ are IID random variables with $\mathbf{P}(X_n=t_i)=p_i$. We refer to such measures as \emph{homogeneous self-similar measures}. After an affine change of coordinates (which does not affect any of the properties we will consider), we can always assume that $t_1=0$ and $t_2=1$.

Recall that the convolution of two Borel probability measures $\mu,\nu$ on $\R$ is defined by the formula
\[
\mu*\nu(A) = (\mu\times\nu)\{ (x,y):x+y\in A \},
\]
and that the Fourier convolution formula $\widehat{\mu*\nu}=\widehat{\mu}\widehat{\nu}$ holds in this context. This extends to convolutions of countably many measures $\mu_n$ if $\sum_n X_n$ converges absolutely for $X_n\in\supp(\mu_n)$.

Recall that $\mu_{a,t}^p$ is an infinite convolution of discrete measures $\sum_{i=1}^m p_i\, \delta_{t_i a^n}$, so that its Fourier transform is given by
\[
\widehat{\mu}_{a,t}^p(u) = \prod_{n=1}^\infty \Phi(a^n u),
\]
where $\Phi(u) =  \Phi_{p,t}(u)= \sum_{j=1}^m p_j \exp(2\pi it_j u).$

\begin{lemma}\label{lem:Phi}
The following holds for all $y\in\R$ and $c\in(0,1)$: if $d(y,\mathbb{Z})>\frac{c}{2},$ then $|\Phi(y)|<1-\eta(c,p)$, where
\[
\eta(c,p)= p_1+ p_2-\sqrt{p_1^2+2p_1p_2\cos(\pi c)+p_2^2}.
\]
\end{lemma}

\begin{proof}
We have that
\begin{align*}
|\Phi(y)|&=|p_1+p_2 e^{2\pi iy}+\sum_{j=3}^m p_j e^{2\pi it_jy}|\\
&\le |p_1+p_2\cos(2\pi y)+i \,p_2\sin(2\pi y)| +\sum_{j=3}^m p_j\\
&= 1-p_1-p_2+\sqrt{p_1^2+2p_1p_2\cos(2\pi y)+p_2^2}.
\end{align*}
Using that $d(y,\mathbb{Z})>\frac{c}{2},$ we obtain that
$
\cos(2\pi y)<\cos(\pi c),
$
so we get the claim.
\end{proof}

\begin{remark}
In the special case of Bernoulli convolutions, $\Phi(u)=\cos(2\pi u)$ and $\eta(c,p)=1-\cos(\pi c).$
\end{remark}

Following Kaufman \cite{Kaufman84}, we use the Erd\H{o}s-Kahane argument to establish quantitative power decay outside of a sparse set of frequencies:
\begin{prop} \label{prop:EK0}
Given $a\in (0,1)$ and a probability vector $p=(p_1,\ldots,p_m)$ there is a constant $C=C_a>0$ such that the following holds: for each $\e>0$ small enough (depending continuously on $a$) the following holds for all $T$ large enough:  the set of frequencies
$u\in[-T,T]$ such that $|\widehat{\mu}_{a,t}^p(u)| \ge T^{-\e}$ can be covered by $C_a T^\delta$ intervals of length $1$, where $C_a>0$ depends only on $a$,
\begin{align}
\delta &= \frac{\log\left(\lceil 1+1/a\rceil\right)\tilde{\e}+h(\tilde{\e})}{\log(1/a)}, \label{eq:def-delta}\\
\tilde{\e} &= \frac{\log(a)}{\log(1-\eta(\tfrac{a}{a+1},p))}\, \e, \notag
\end{align}
and $h(\tilde{\e})=-\tilde{\e}\log(\tilde{\e})-(1-\tilde{\e})\log(1-\tilde{\e})$ is the entropy function. 
\end{prop}
\begin{proof}
Choosing $N\in\mathbb{N}$ such that $a^{1-N}\le T<a^{-N}$ we may assume that $T=a^{-N}$.

First we consider $u$ such that $0\le u\le a^{-N}.$ Then we can write $u=t a^{-N}$ with $t\in[0,1].$

We have that
\begin{align*}
|\widehat{\mu}_{a,t}^p(u)| &\le |\prod_{j=1}^\infty \Phi(a^j u)|\\
&= |\prod_{j=1}^\infty \Phi(a^j a^{-N}t)|\\
&\le \prod_{j=1}^N |\Phi(a^{j-N}t)|\\
&=\prod_{j=0}^{N-1} |\Phi(a^{-j}t)|.
\end{align*}

We denote the distance of $y\in\R$ to the closest integer by $\|y\|$. Given $\e>0$, we let $\wt{\e}$ be as in the statement. Let
\[
S(N, \tilde{\varepsilon}):=\left\{t\in[0,1]\colon \|a^{-j} t\|< \xi \mbox{ for at least } (1-\tilde{\e})N \mbox{ integers } j\in[N]\right\},
\]
where we denote $[N]=\{0,1,\ldots,N-1\}$, and $\xi=\xi(a)=\tfrac{a}{2(a+1)}$.

We observe that if $t\notin S(N, \tilde{\varepsilon})$ then, by Lemma \ref{lem:Phi},
\[
|\widehat{\mu}_{a,t}^p(u)|\le (1-\eta(2\xi,p))^{\tilde{\varepsilon} N}=  a^{N\e} = T^{-\e},
\]
using the definition of $\wt{\e}$. We deduce that
\begin{equation} \label{eq:decay-on-SN}
\{t\in[0, 1]\colon |\widehat{\mu}_{a,t}^p(ta^{-N})|\ge T^{-\e}\}\subseteq S(N, \tilde{\varepsilon}).
\end{equation}
Hence, it order to prove that $\{u\in[0, T]\colon |\widehat{\mu}_{a,t}^p(u)|\ge T^{-\e}\}$ can be covered by a small number of intervals of length
1, we will estimate the amount and length of intervals needed to cover $S(N, \tilde{\varepsilon})$.

For each $t\in [0,1]$, we define integers $r_j(t)$ and $\e_j(t)\in [-1/2, 1/2)$ such that
\[
a^{-j}t= r_j(t)+\e_j(t)
\]
so that $t\in S(N, \tilde{\varepsilon})$ precisely when $|\e_j(t)|<\xi$ at least $(1-\tilde{\e})N$ times among indices $j\in[N]$.
We will simply write $r_j$ and $\e_j$ when no confusion arises.

Let $N_1=\lceil (1-\tilde{\e}) N\rceil$. For each $t\in S(N, \tilde{\varepsilon})$, there is a subset $I\subset [N]$ with at least $N_1$ elements such that $|\e_j|<\xi$ for all $j\in I$. We will estimate the size of $S(N,\tilde{\e})$ by considering each index set $I$ separately, and for this we define 
\[
S(I,\tilde{\e}) = \{t\in[0,1]\colon \|a^{-j} t\|< \xi \text{ for all } j\in I\}.
\]

We have that $t=r_0+\varepsilon_0$, so for $t\in[0,1]$ there are at most $2$ choices for $r_0$. On the other hand, since $r_{j+1}+\e_{j+1}=a^{-1}(r_j+\e_j)$ for $j\ge 0$, we have that
\[
|r_{j+1}-a^{-1} r_j|\le |\e_{j+1}|+a^{-1}|\e_j|.
\]
Using that $-1/2\le \e_j, \e_{j+1}<1/2,$ we obtain that each value of $r_j$ can be followed by at most $\lceil 1+1/a\rceil$ choices of $r_{j+1}.$

If $j,j+1\in I$, then $|\e_{j+1}|,|\e_j|<\tfrac{a}{2(a+1)}$ so that $|\e_{j+1}|+\frac{1}{a}|\e_j|<1/2$, and at
most one value of $r_{j+1}$ is possible. Note that
\[
|\{ j\in [N]:j,j+1\in I\}| \ge N-2|N\setminus N_1|-1 \ge (1-2\widetilde{\e})N-1.
\]
Thus, the total number of sequences $r_1, \dots, r_N$ can be bounded by
\[
M_N := \left(\lceil 1+1/a\rceil\right)^{2\tilde{\e}N+1}.
\]
In particular, $r_N$ can take at most $M_N$ values. Since
\[
t\in (a^N r_N-a^N/2, a^N r_n+a^N/2),
\]
we obtain that $S(I,\tilde{\e})$ can be covered by $M_N$ intervals of length $a^N$.

Now, estimating $\binom{N}{N_1}$ by Stirling's formula, we see that the number of index sets $I$ is at most $\exp(h(\tilde{\e}) N)$ for large enough $N$.

We conclude that $S(N, \tilde{\varepsilon})$ can be covered by
\[
 \left(\lceil 1+1/a\rceil\right)^{2\tilde{\e}N+1} e^{h(\tilde{\e})N}
\]
intervals of length $a^N$. Recalling \eqref{eq:decay-on-SN} and rescaling, we see that $\{ u\in [0,T]: |\widehat{\mu}_{a,t}^p(u)|> T^{-\e}\}$ can be covered by the claimed number of intervals of length $1$. The situation for $u\in [-T,0]$ is completely analogous, so this finishes the proof.

\end{proof}

\section{A generalization of a theorem of Kaufman}

We now apply Proposition \ref{prop:EK0} to  obtain a variant of a result of Kaufman \cite{Kaufman84}:
\begin{thm}\label{thm:kaufman}
Let $F\in C^2(\R)$ such that $F^{\prime\prime}>0$ and let $\mu=\mu_{a,t}^p$ be a (homogeneous) self-similar measure on $\R$ which is not a single atom. Then there exist $\sigma=\sigma(\mu)>0$ (independent of $F$) and $C=C(F,\mu)>0$ such that
\[
|\widehat{F\mu}(u)| \le  C |u|^{-\sigma}.
\]
\end{thm}

We underline that the value of $\sigma$ is effective. For the proof we need the following well known result, see \cite[Proposition 2.2]{FengLau09}.
\begin{prop}\label{prop}
Let $\mu$ be a self-similar measure on the real line which is not a single atom. Then there exist $C=C(\mu)>0$ and $s=s(\mu)>0$ such that $\mu(B(x,r))\le C\, r^s$ $\forall x, \forall r>0$.
\end{prop}
We also recall that, if the underlying iterated function system satisfies the open set condition, then one can take
\[
s=  \min_{i=1}^m \frac{\log p_i}{\log a}.
\]
In particular, if $p_i=1/m$ for all $i$, then $s=\log m/\log(1/a)$.

\begin{proof}[Proof of Theorem \ref{thm:kaufman}]
For simplicity we only consider the case $u>0$, with the case $u<0$ being exactly analogous. Throughout the course of the proof, $C$ denotes a positive constant that is allowed to depend only on $a,p,t$ and $F$, and may change from line to line. By $x\approx y$ we mean that $C^{-1} x \le y \le C x$.

Fix $u\gg 1$. Then there exists $N=N(u)\in\N$ such that
\begin{equation}\label{eq-1}
1< a^N u^{2/3}\le a^{-1}.
\end{equation}
In other words, $N = \lfloor\frac{\log(u^{-2/3})}{\log(a)}\rfloor$.

We can write
\[
\mu = \mu_N * \lambda_N,
\]
where $\mu_N = *_{n=1}^N (\sum_{i=1}^m p_i \delta_{a^n t_i})$ is the step $n$ discrete approximation to $\mu$, and $\lambda_N$ is a copy of $\mu$ scaled down by a factor $a^N$. Then

\begin{equation}\label{conv}
\widehat{\mu}(u)=  \widehat{\mu_N}(u) \widehat{\lambda_N}(u)= \prod_{n=1}^{N}\Phi(a^n u) \prod_{n=N+1}^{\infty}\Phi(a^n u).
\end{equation}

Since $F$ is $C^2$, we have that $F(x+y)= F(x)+F^{\prime}(x)y+O(y^2)$. Then
\[
uF(x+y)= uF(x)+uF^{\prime}(x)y+O(uy^2).
\]
Write $e(x)=e^{-2\pi i x}$ for simplicity. Using that $|e(\delta)-1|=O(\delta)$, we estimate
\begin{align*}
\widehat{F\mu}(u) &= \int e^{-2\pi i uF}\, d\mu\\
&= \iint e(uF(x+y))\, d\mu_N(x) d\lambda_N(y)\nonumber\\
&= \iint e(uF(x)+uF^{\prime}(x)y)(1+O(u y^2))\, d\mu_N(x) d\lambda_N(y)\\
&= \int e(uF(x)) \left( \int e(uF^{\prime}(x)y)  d\lambda_N(y) \right)d\mu_N(x) +O(u(a^N)^2).
\end{align*}
Since that $a^N\approx u^{-2/3},$ we see that $u (a^N)^{2}=O(u^{-1/3})$. Using that $|\widehat{\lambda_N}(\xi)|=|\widehat{\mu}(a^N\xi)|$ for all $\xi$, we obtain
\begin{align*}
|\widehat{F\mu}(u)| &\le \left| \int e(uF(x)) \left( \int e(uF^{\prime}(x)y)  d\lambda_N(y) \right)d\mu_N(x) \right| + O(u^{-1/3})\\
&\le \int |\widehat{\lambda_N}(uF^{\prime}(x))|\, d\mu_N(x) +O(u^{-1/3})\\
&= \int |\widehat{\mu}(ua^NF^{\prime}(x))|\, d\mu_N(x) +O(u^{-1/3}).
\end{align*}

Let $A:=\sup_{x\in \supp(\mu)} |F^{\prime}(x)|$ and $T=Aa^Nu>0$, and note that $T\approx u^{1/3}$. Fix $\e>0$ to be determined later. Then, by Proposition \ref{prop:EK0}, there is $C=C(\mu)>0$ such that for large enough $u$, the set of frequencies $|\xi|\in[0,T]$ such that $|\widehat{\mu}(\xi)| \ge T^{-\e}$ can be covered by $CT^\delta$ intervals of length $1$, where
$\delta= \frac{\log\left(\lceil 1+1/a\rceil\right)\tilde{\e}+h(\tilde{\e})}{\log(1/a)},$ and
$\tilde{\e}= \frac{\log(a)}{\log(1-\eta(\tfrac{a}{a+1},p))}\, \e$.  Let $I_1,\dots, I_{CT^{\delta}}$ be these intervals. Then $|\widehat{\mu}(\xi)|\le T^{-\e}$ for all $\xi\notin\cup_{j=1}^{CT^{\delta}}I_j$.

We consider
\[
\Gamma:=\left\{x\colon a^NuF^{\prime}(x)\in\bigcup_{j=1}^{CT^{\delta}}I_j\right\}.
\]
Thus
\[
\int |\widehat{\mu}(ua^NF^{\prime}(x))|\, d\mu_N(x)=  \int_{\Gamma}
+ \int_{\Gamma^c}\le\mu_N(\Gamma)+ T^{-\e} \le \mu_N(\Gamma) + O(u^{-\e/3}).
\]
To conclude we will prove that there exists $\gamma>0$ such that $\mu_N(\Gamma)\le u^{-\gamma}$. Since $\mu=\mu_N*\lambda_N$, and the support of $\lambda_N$ is contained in an interval $[-C a^N, C a^N]$ for some $C=C(a,t)$, we see that for any interval $I=(b_1,b_2)$,
\begin{equation}\label{prop-N}
\mu_N(I) \le \mu(b_1-C a^N,b_2+C a^N).
\end{equation}
On the other hand, $\Gamma=\cup_{j=1}^{CT^{\delta}}(F^\prime)^{-1} J_j$, where $J_j$ are intervals of length $|J_j|=(a^Nu)^{-1}\approx a^{N/2}$. By our assumption that $F''>0$, we see that $\Gamma$ is covered by $C T^\delta$ intervals $J'_j$ of length $C a^{N/2}$.
We observe that the unique constant $C$ which depends on $F$ is the last one, that is, the one which modifies the lengths of the
intervals. Using \eqref{prop-N} we get $\mu_N(J'_j)\le C a^{\frac{Ns}2}$, and then
\[
\mu_N(\Gamma)\le C T^{\delta}  a^{\frac{Ns}2}\le C  u^{\frac{(\delta-s)}{3}}.
\]
At this point we assume that $\e$ was taken small enough that $\delta<s(\mu)$, and conclude that
\[
|\widehat{F\mu}(u)| \le C u^{\frac{(\delta-s)}{3}}+ C u^{-\e/3}+  C u^{-1/3} \le  C u^{-\min\left\{\frac{(s-\delta)}{3}, \frac{\e}{3}\right\}}.
\]
\end{proof}

As an example, we obtain that if $\mu$ is the Cantor-Lebesgue measure on the middle-thirds Cantor set, then even though $\widehat{\mu}(u)$ does not decay as $u\to\infty$, for $\widehat{F\mu}$ we have a uniform explicit decay:

\begin{cor} \label{cor:Fmu}
Let $\mu$ be the Cantor-Lebesgue measure. Then for every $C^2$ function $F:\R\to\R$ such that $F''>0$ there exists a constant $C_F>0$ such that
\[
|\widehat{F\mu}(u)| \le C_F |u|^{-\sigma}
\]
Moroever, one can take $\sigma=0.016$.
\end{cor}
\begin{proof}
The non-quantitative statement is immediate from Theorem \ref{thm:kaufman}. We indicate the calculations required to obtained the value $\sigma=0.016$. Going through the proof of Theorem \ref{thm:kaufman}, notice that $\sigma= \min\left\{\frac{(s-\delta)}{3}, \frac{\e}{3}\right\}$.  We want to find $\e$ such that $s(\mu)-\delta(\e)=\e.$ In this case we have that $a=1/3, p=(1/2,1/2), \eta(c,p)=1-\cos(\pi c)$ and
$s(\mu)=\frac{\log(2)}{\log(3)}.$ So, using \eqref{eq:def-delta}, we obtain
\[
\delta({\tilde{\e}})=\frac{2\log(2)\tilde{\e}+ h(\tilde{\e})}{\log(3)}, \quad \mbox{ and } \tilde{\e}=\frac{2\log(3)\e}{\log(2)}.
\]
If we let
\begin{align*}
G(\e):=s(\mu)-\delta(\e)-\e&= \frac{\log(2)}{\log(3)}-5\e+\frac{2}{\log2}\e\log\left(\frac{2\log(3)\e}{\log(2)}\right)+\\
&\left(\frac{\log(2)-2\log(3)\e}{\log(2)\log(3)}\right)\log\left(\frac{\log(2)-2\log(3)\e}{\log(2)}\right),
\end{align*}
then  $G(\e)$ has domain $(0, \frac{\log2}{\log9})$, and $G(\e)=0$ if and only if $\e= 0.048279\ldots$, giving the claimed value of $\sigma$.
\end{proof}

\section{$L^2$ dimension of convolutions}
\label{sec:L2}

We begin by recalling the definition of $L^q$ dimensions. Let $q\in (1,+\infty)$, and set $s_n(\mu,q)= \sum_{Q\in\mathcal{D}_n} \mu(Q)^q$, with $(\mathcal{D}_n)$ the partition of $\R^d$ into dyadic intervals of length $2^{-n}$. Define
\[
\dim_q(\mu):=\liminf_{n\to+\infty} \frac{\log(s_n(\mu,q))}{(q-1)\log(2^{-n})}.
\]
The $L^2$ dimension of a measure is also known as \emph{correlation dimension}. Note also that the Frostman exponent $\dim_\infty$ can also be defined as
\[
\dim_\infty(\mu):=\liminf_{n\to +\infty}\frac{\log \max\{\mu(Q):Q\in\mathcal{D}_n\}}{\log(2^{-n})}.
\]
It is well known that the function $q\mapsto \dim_q(\mu)$ is continuous and non-increasing on $(1,+\infty]$ and that $\dim_q(\mu) \le \hdim(\mu)$ for any $q\in (1,+\infty]$, where $\hdim$ is the lower Hausdorff dimension of a measure, defined as
\[
\hdim(\mu):=\inf\{ \hdim(A):\mu(A)>0\}.
\]
We refer the reader to \cite{FLR02} for the proofs of these facts and further background on dimensions of measures.

We prove next that homogeneous self-similar measures strictly increase $L^2$ dimension of measures upon convolution, in a quantitative sense. A version of this result with far worse quantitative estimates could also be inferred from  \cite[Theorem 2.1]{Shmerkin16}.
\begin{thm} \label{thm:flattening}
Let $\mu=\mu_{a,p}^t$ be as above. Given any $\kappa>0$, there is $\sigma=\sigma(a,p,\kappa)>0$ such that the following holds: let $\nu$ be any Borel probability measure with $\dim_2(\nu) \le 1-\kappa$. Then
\[
\dim_2(\mu*\nu) > \dim_2(\nu)+\sigma.
\]
More precisely, one can take $\sigma=2\e$, where $\e=\e(a,p,\kappa)$ is such that the value of $\delta=\delta(\e,a,p)$ given in \eqref{eq:def-delta} of Proposition \ref{prop:EK0} satisfies
\begin{equation} \label{eq:def-epsilon}
\kappa - 2\e = \delta.
\end{equation}
\end{thm}

\begin{remark}
We need to assume that $\dim_2(\nu)\le 1-\kappa$ because if $\dim_2(\nu)$ is already $1$ or very close, it cannot grow when we convolve with $\mu$.
\end{remark}

\begin{proof}
To begin, we note that there is a value of $\e\in (0,1/2)$  satisfying \eqref{eq:def-epsilon} by standard continuity arguments (similar to the calculation in Corollary \ref{cor:Fmu}) that we omit. From  \cite[Lemma 2.5]{FNW02}, we know that $\dim_2(\eta)=1-\alpha_2(\eta)$, where
\[
\alpha_2(\eta) = \limsup_{T\to\infty} \frac{\log \int_{|\xi|\le T} |\widehat{\eta}|^2 d\xi}{\log T}.
\]
So we have to show that if $\alpha_2(\nu)\ge\kappa$, then $\alpha_2(\mu*\nu)<\alpha_2(\nu)-\sigma$. We may assume $T=2^N,$ for some $N\in\N.$

Let $\kappa_0 = \alpha_2(\nu)\ge \kappa$. By definition of $\alpha_2$, for any $\e_0>0$,
\[
\int_{ |\xi| \le 2^N}|\widehat{\nu}(\xi)|^2 d\xi \le O_{\e_0}(1) 2^{N(\kappa_0+\e_0)}.
\]
Let
\begin{align*}
E_N &= \{ \xi:  |\xi| \le 2^N, |\widehat{\mu}(\xi)| \le 2^{-\e N}\},\\
F_N &= \{ \xi:  |\xi| \le 2^N, |\widehat{\mu}(\xi)| > 2^{-\e N}\}.
\end{align*}
We know from Proposition \ref{prop:EK0} that $F_N$ can be covered by $C_a  2^{\delta N}$ intervals of length $1$
so it has measure at most $C_a 2^{\delta N}$. Here $\delta=\delta(\mu,\e)>0$ is the value given in Proposition \ref{prop:EK0}.

Using all this, we have
\begin{align*}
\int_{|\xi| \le 2^N} |\widehat{\mu*\nu}(\xi)|^2 \,d\xi &= \int_{E_N\cup F_N} |\widehat{\mu}(\xi)|^2|\widehat{\nu}(\xi)|^2\,d\xi \\
&\le \int_{E_N} 2^{-2\e N} |\widehat{\nu}(\xi)|^2 \,d\xi+ \int_{F_N} 1 \,d\xi\\
&\le O_{\e_0}(1) 2^{-2\e N} 2^{(\kappa_0+\e_0)N} + C_a 2^{\delta N}\\
&\le O_{\e_0,a}(1) 2^{(\kappa_0-2\e+\e_0)N},
\end{align*}
using that $\kappa_0\ge \kappa$ and the definition of $\e$ in the last line. Since this holds for all $\e_0>0$, it follows from the definition of $\alpha_2$ that
\[
\alpha_2(\mu*\nu) \le \kappa_0 - 2\e,
\]
which gives the claim since $\sigma=2\e$.
\end{proof}

\section{Dimension of Bernoulli convolutions}

Recall that $\nu_a^p$ denotes the (biased) Bernoulli convolution with contraction ratio $a$ and weight $p$. We next prove Theorem \ref{thm:dim-Bernoulli-intro}, which we state again:
\begin{thm}\label{dim-Bernoulli}
For every $p_0\in (0,1/2)$ there is $C=C(p_0)>0$ such that
\[
 \inf_{p\in [p_0,1-p_0]}\dim_\infty(\nu_a^p) \ge 1 - C (1-a) \log(1/(1-a)).
\]
\end{thm}

For the proof we need the following relation between $L^2$ and $L^\infty$ dimensions. Although it is a simple consequence of Young's inequality, we have not been able to find it in the literature.
\begin{lemma} \label{lem:young}
Let $\mu,\nu$ be probability measures on $\R$. Then
\[
\dim_\infty(\mu*\nu) \ge \frac{\dim_2(\mu)+\dim_2(\nu)}{2}.
\]
\end{lemma}
\begin{proof}
Let $\mu_n$, $\nu_n$ denote the absolutely continuous measures with locally constant density equal to $2^n\mu(Q)$, $2^n\nu(Q)$ on each interval $Q\in\mathcal{D}_n$, respectively. We also denote by $\mu_n,\nu_n$ the respective densities. Note that $\mu_n(Q)=\mu(Q)$, and
\[
\|\mu_n\|_2^2 = 2^{n} s_n(\mu,2)
\]
and likewise for $\nu_n$. By Young's inequality,
\begin{equation} \label{eq:young}
\|\mu_n*\nu_n\|_\infty \le 2^{n} s_n(\mu,2)^{1/2} s_n(\nu,2)^{1/2}.
\end{equation}
On the other hand, for any $Q\in\mathcal{D}_n$, we have
\begin{align*}
\mu*\nu(Q) &\le  \sum_{ J,J'\in\mathcal{D}_n: Q\cap J+J'\neq\varnothing} \mu(J)\nu(J') \\
&= \sum_{ J,J'\in\mathcal{D}_n: Q\cap J+J'\neq\varnothing} \mu_n(J)\nu_n(J') \\
&\le \mu_n*\nu_n(5Q).
\end{align*}
where $5Q$ is the interval with the same center as $Q$ and $5$ times the length (this follows since $J+J'\subset 5Q$ whenever $Q\cap J+J'\neq\varnothing$). Combining this with \eqref{eq:young}, we conclude that
\[
\sup\{\mu*\nu(Q):Q\in\mathcal{D}_n\} \le 5 s_n(\mu,2)^{1/2} s_n(\nu,2)^{1/2}.
\]
Taking logarithms, dividing by $\log (2^{-n})$, and taking $\liminf$ on both sides, we get the claim.
\end{proof}

\begin{proof}[Proof of Theorem \ref{dim-Bernoulli}]
Fix $a\in (0,1)$ close to $1$. We define
\[
N=N_a:=\min \{n\in\N\colon a^n<1/2\}.
\]
Then,
\[
\frac{a}{2}\le a^N<\frac{1}{2}.
\]
In particular, assuming as we may that $a>1/2$, we see that $a^N\in (1/4,1/2)$.

Fix $\kappa\in (0,1)$, and suppose that $\dim_2(\nu_a^p)\le 1-\kappa$. Let us write $S_a(x)=ax$ for the map that scales by $a$, and recall that
\begin{equation} \label{eq:convolution-structure}
\nu_a^p = \nu_{a^N}^p * S_a \nu_{a^N}^p * \cdots * S_{a^{N-1}} \nu_{a^N}^p.
\end{equation}
We know that $\dim_2(\nu_{a^N}^p)\ge 0$. In fact, since the associated IFS satisfies the open set condition,
\begin{equation} \label{eq:corr-dim-OSC}
\dim_2(\nu_{a^N}^p) = \frac{\log(p^2+(1-p)^2)}{\log(a^N)}.
\end{equation}
Now, using that $\dim_2(\nu_a^p)\le 1-\kappa$ and \eqref{eq:convolution-structure},
we get $\dim_2(\nu_{a^N}^p)\le 1-\kappa$. By Theorem \ref{thm:flattening}, there is $\sigma=\sigma_{a,p}(\kappa)>0$ such that
\[
\dim_2(\nu_{a^N}^p *S_a \nu_{a^N}^p)\ge \sigma.
\]
Proceeding inductively according to \eqref{eq:convolution-structure}, after $N-1$ steps we obtain that if $\dim_2(\nu_a^p)\le 1-\kappa$, then
\[
\dim_2(\nu_a^p) \ge (N-1)\sigma.
\]
It follows that if $\kappa$ is such that $\sigma=\sigma_{a,p}(\kappa)=1/(N-1)$, then
\[
\dim_2(\nu_a^p)\ge 1-\kappa.
\]
Thus, it remains to estimate such $\kappa$. We denote by $C_i$ positive constants that only depend on $p_0$. By \eqref{eq:def-epsilon}, we have that $\kappa = \delta+\sigma$, where $\delta=\delta(\sigma/2)$ is given by \eqref{eq:def-delta}. Note that $\tilde{\e}=C(a^N,p)\sigma$, where $C>0$ depends continuously on $a^N$ and $p$. Since $a^N\in [1/4,1/2]$ and $p\in [p_0,1-p_0]$, a calculation using \eqref{eq:def-delta} shows that there is a constant $C_1$ such that $\delta\le C_1 \sigma\log(1/\sigma)$ provided $\sigma$ is small enough (which we may assume).  We deduce that
\begin{equation} \label{eq:final-lower-bound}
\dim_2(\nu_a^p)\ge 1-\kappa \ge 1-\sigma-C_1\sigma\log(1/\sigma) \ge 1- C_2\sigma\log(1/\sigma)
\end{equation}
if $\sigma$ is small enough. On the other hand, since $a^{1/\sigma}=a^{N-1}<a^{-1}/2<2/3$ (say), we have that $\sigma \le \log(1/a)/\log(3/2)$. Finally, using that $\log(1/a) \le 2(1-a)$ for $1-a$ small, we deduce that
\[
\sigma \le C_4 (1-a).
\]
Together with \eqref{eq:final-lower-bound}, this yields
\[
\dim_2(\nu_a^p) \ge  1 - C_5 (1-a) \log(1/(1-a)).
\]
Since we have the decomposition
\[
 \nu_a^p = \nu_{a^2}^p * S_a\nu_{a^2}^p,
\]
and scalings do not change $L^2$ dimension, we can appeal to Lemma \ref{lem:young} to conclude that
\[
\dim_\infty(\nu_a^p) \ge 1 - C_5 (1-a^2) \log(1/(1-a^2)) \ge 1- C_6(1-a)\log(1/(1-a)).
\]
\end{proof}

When $p_0<1/2$, the estimate \eqref{eq:corr-dim-OSC} is away from $1$ so it does not help in improving the estimate given by the theorem. However, for unbiased Bernoulli convolutions, already \eqref{eq:corr-dim-OSC} is very close to $1$ so that we are able to obtain an improved lower bound in this case:
\begin{cor} \label{cor:unbiased}
There is an absolute constant $C>0$ such that
\[
\dim_\infty(\nu_a) \ge 1 - C(1-a)^2 \log(1/(1-a)).
\]
\end{cor}
\begin{proof}
Again, fix $a\in(0,1)$ close to $1$ and let $N=\inf\{n\in\N\colon a^n\le 1/2\}$, so that $a^N > a/2$.

According to \eqref{eq:corr-dim-OSC},
\begin{align*}
\dim_2(\nu_{a^N}) &= \frac{\log(1/2)}{\log a^N} \\
&\ge \frac{\log (1/2)}{\log( a/2)} = 1 -  \frac{\log(1/ a)}{\log(2/ a)}.
\end{align*}

Proceeding as in the proof of the theorem \ref{dim-Bernoulli}, we obtain that if $\dim_2(\nu_a)\le 1- \kappa$ then there exists
$\sigma=\sigma(\kappa, a)>0$ such that
\begin{align*}
\dim_2(\nu_a)&\ge \dim_2(\nu_{ a^N})+ (N-1)\sigma\\
&\ge 1-\frac{\log(1/ a)}{\log(2/ a)}+ (N-1)\sigma.
\end{align*}
Now if $\kappa$ is such that $\sigma=\frac{\log(1/ a)}{\log(2/ a)}\frac{1}{N-1},$ then
\[
\dim_2(\nu_a)\ge 1-\kappa.
\]
Then we want to estimate such $\kappa$. Proceeding as in the proof of the above theorem, we get
\[
\dim_2(\nu_a)\ge 1-\kappa\ge 1-c_1\sigma\log(1/\sigma).
\]
On the other hand, using that $1/4< a^N<1/2$ and that $\log(1/ a) \le 2(1- a)$ for $1- a$ small, we obtain
 $\frac{1}{N-1}\le c_2(1- a)$ and then
 \[
 \sigma\le c_3(1- a)^2.
 \]
Thus
 \[
\dim_2(\nu_a) \ge 1 - c_4(1-a)^2 \log(1/(1-a)).
\]
Invoking Lemma \ref{lem:young} as in the proof of Theorem \ref{dim-Bernoulli} finishes the proof.
\end{proof}


\end{document}